\documentclass[12pt]{amsart}
\usepackage{mathdots}
\usepackage[margin=1.25in]{geometry}

\usepackage{amsmath}
\usepackage{amsfonts}
\usepackage{amssymb}
\usepackage{amscd}
\usepackage{graphicx}
\usepackage[abbrev,alphabetic]{amsrefs}
\usepackage[dvipsnames]{xcolor}
\usepackage{soul}
\setstcolor{red}
\usepackage{stmaryrd}
\usepackage{multicol}

\newcommand\hidden[1]{}
\usepackage{mathtools}
\newtagform{hidden}[\hidden]{}{}

\usepackage{hyperref}

\usepackage{amsthm}
\usepackage{comment}
\usepackage[all,cmtip]{xy}
\usepackage{tikz-cd}
\usetikzlibrary{cd}
\usetikzlibrary{arrows.meta}

\definecolor{ink}{HTML}{263B48}
\definecolor{accent}{HTML}{365E71}
\definecolor{singularpoint}{HTML}{A65F50}
\definecolor{paperwash}{HTML}{F5F7F8}

\usepackage{multicol}

\makeatletter
\def\@tocline#1#2#3#4#5#6#7{\relax
  \ifnum #1>\c@tocdepth 
  \else
    \par \addpenalty\@secpenalty\addvspace{#2}
    \begingroup \hyphenpenalty\@M
    \@ifempty{#4}{
      \@tempdima\csname r@tocindent\number#1\endcsname\relax
    }{
      \@tempdima#4\relax
    }
    \parindent\z@ \leftskip#3\relax \advance\leftskip\@tempdima\relax
    \rightskip\@pnumwidth plus4em \parfillskip-\@pnumwidth
    #5\leavevmode\hskip-\@tempdima
      \ifcase #1
       \or\or \hskip 1em \or \hskip 2em \else \hskip 3em \fi
      #6\nobreak\relax
    \hfill\hbox to\@pnumwidth{\@tocpagenum{#7}}\par
    \nobreak
    \endgroup
  \fi}
\makeatother

\hypersetup{
bookmarksdepth=3,
bookmarksopen,
bookmarksnumbered,
pdfstartview=FitH,
colorlinks,
linkcolor=Sepia,
anchorcolor=BurntOrange,
citecolor=MidnightBlue,
citecolor=OliveGreen,
filecolor=BlueViolet,
menucolor=Yellow,
urlcolor=OliveGreen
}

\newsavebox{\pullback}
\sbox\pullback{
\begin{tikzpicture}
\draw (0,0) -- (1ex,0ex);
\draw (1ex,0ex) -- (1ex,1ex);
\end{tikzpicture}}

\newsavebox{\pullbackdl}
\sbox\pullbackdl{
\begin{tikzpicture}
\draw (-1ex,0ex) -- (0ex,0ex);
\draw (0ex,-1ex) -- (0ex,0ex);
\end{tikzpicture}}

\newsavebox{\pushoutdr}
\sbox\pushoutdr{
\begin{tikzpicture}
\draw (-1ex,-1ex) -- (-1ex,0ex);
\draw (-1ex,0ex) -- (0ex,0ex);
\end{tikzpicture}}

\newcommand{\stacksproj}[1]{{\cite{stacks-project}*{Tag~{#1}}}}

\newcommand{\bF}{\mathbb{F}}

\newcommand{\bP}{\mathbb{P}}

\newcommand{\cF}{\mathcal{F}}

\newcommand{\cJ}{\mathcal{J}}
\newcommand{\cI}{\mathcal{I}}

\newcommand{\cO}{\mathcal{O}}
\newcommand{\sO}{\mathcal{O}}

\DeclareMathOperator{\Spec}{Spec}

\theoremstyle{plain}
\newtheorem{theorem}{Theorem}[section]

\newtheorem{proposition}[theorem]{Proposition}
\newtheorem{proposition-definition}[theorem]{Proposition-Definition}
\newtheorem{lemma}[theorem]{Lemma}

\newtheorem*{claim*}{Claim}

\newtheorem{theoremIntro}{Theorem}

 \newtheorem*{theoremB}{Theorem B}

\theoremstyle{definition}
\newtheorem{definition}[theorem]{Definition}

\newtheorem*{setup*}{Setup}

\theoremstyle{remark}
\newtheorem{remark}[theorem]{Remark}

\numberwithin{equation}{theorem}



\newif\ifshowColoursAndTodoes

\showColoursAndTodoestrue

\ifshowColoursAndTodoes
\def\todo#1{\textcolor{Mahogany}%
{\footnotesize\newline{\color{Mahogany}\fbox{\parbox{\textwidth-15pt}{\textbf{todo: } #1}}}\newline}}
\def\commentbox#1{\textcolor{Mahogany}%
{\footnotesize\newline{\color{Mahogany}\fbox{\parbox{\textwidth-15pt}{\textbf{comment: } #1}}}\newline}}

\else
\def\todo#1{}
\def\commentbox#1{}
\renewcommand{\st}[1]{}
\colorlet{red}{black!100}
\colorlet{teal}{black!100}
\colorlet{brown}{black!100}
\colorlet{blue}{black!100}
\colorlet{magenta}{black!100}
\colorlet{purple}{black!100}
\colorlet{cyan}{black!100}
\fi

\RequirePackage[normalem]{ulem}
\DeclareRobustCommand{\DIFdel}[1]{{\color{red}\ifmmode
  \text{\sout{\ensuremath{\displaystyle #1}}}\else\sout{#1}\fi}}
\newcommand{\DIFreplacementsep}{\ifmmode\;\else\ \fi}
\providecommand{\DIFaddbegin}{}
\providecommand{\DIFaddend}{}

\title[An irregular smooth Fano fourfold]{An irregular smooth Fano fourfold in positive characteristic}
\author{Joe Waldron}
\address{Michigan State University, Department of Mathematics, Wells Hall, 619 Red Cedar Road, East Lansing, MI 48824, USA}
\email{waldro51@msu.edu}
\author{Jakub Witaszek}
\address{Northwestern University, Department of Mathematics, Lunt Hall, 2033 Sheridan Road, Evanston, IL 60208, USA}
\email{jakub.witaszek@northwestern.edu}
\date{}
\hypersetup{pdftitle={An irregular smooth Fano fourfold},
            pdfauthor={Joe Waldron and Jakub Witaszek}}


\newtheorem{construction}[theorem]{Construction}

\begin{document}

\begin{abstract}
We construct an explicit smooth projective Fano fourfold $X$ over an algebraically closed field $k$ of
characteristic two with $H^1(X,\cO_X)\simeq k$. Moreover, $X$ does not admit any free rational curve, and so it is not separably uniruled. The example has been constructed using ChatGPT.
\end{abstract}

\maketitle

\section{Introduction}\label{sec:introduction}

Over a field of characteristic zero, Kodaira vanishing implies that $H^i(X,\cO_X)=0$ for all $i>0$ when $X$ is a  smooth
projective Fano variety. Although Kodaira
vanishing can fail in positive characteristic, even for smooth Fano varieties \cites{HaboushLauritzen,LauritzenRao,KovacsCanonical,Totaro}, none of these examples has
nonzero higher cohomology of the structure sheaf. Whether a smooth Fano variety with $H^i(X,\cO_X)\neq 0$ for $i>0$
exists has been an open question considered by many authors (cf.\
\cite{KovacsCanonical}). This question also relates to the study of rational
points on Fano varieties over finite fields (see \cite{Esnault}).

In this article, we construct what is, to the best of our knowledge, the first example
in which this vanishing fails.

\begin{theoremIntro}\label{thm:main}
Over $k=\overline{\bF}_2$, there exists a smooth projective integral
Fano fourfold $X$ such that $H^1(X,\cO_X)\simeq k$.
\end{theoremIntro}

This example has the smallest possible dimension: over algebraically closed
fields, the vanishing always holds in dimensions at most three (see \cites{ShepherdBarron, KawakamiTanakaI, KawakamiTanakaII}). Also, for previously constructed non-smooth examples we refer to \cites{KollarNonrational, SchroerWeakDelPezzo, Maddock}.  \\

Note that the above non-vanishing automatically implies that $X$ is not separably rationally connected (see 
\cite{Gounelas}), thereby also answering a question
of Koll\'ar (cf. 
\cite{Zhu}*{Question~2}). However, more is true: our example admits no free rational
curve and is therefore not separably uniruled.

\begin{theoremB}
The fourfold $X$ admits no free rational curve, and hence is not separably
uniruled.
\end{theoremB}

\DIFaddend The idea of the construction is as follows. First, we choose $\theta\in\bF_{128}\subset k$
of degree seven over $\bF_2$, and define the polynomials in seven variables $x_0,\ldots,x_5, t$:
\begin{align*}
 Q_6&:=x_0x_1+x_2x_3+x_4x_5+\theta x_0^2+\theta^2x_1^2+\theta^3x_2^2+\theta^4x_3^2+\theta^5x_4^2+\theta^6x_5^2,\\
 Q&:=Q_6+t^2,\\
 C&:=x_0^2x_1+x_0x_1^2+x_2^2x_3+x_2x_3^2
       +x_4^2x_5+x_4x_5^2.
\end{align*}
Set $H:=V(Q,C)\subset\bP^6$. Although $H$ is singular, it lies in a
smooth open subset $U$ of $T:= V(C)$ carrying a regular rank-one foliation $\cF$ induced by (see Proposition \ref{prop:foliation}):
\[
 D:=
  \sum_{i=0}^5x_i^2\frac{\partial}{\partial x_i}.
\]
We form the smooth quotient $Z:=U/\cF$.
The invariant divisor  $H$ descends to a Cartier divisor $X\subset Z$,
and the restricted morphism $H\to X$ is the quotient by the induced
derivations.\\

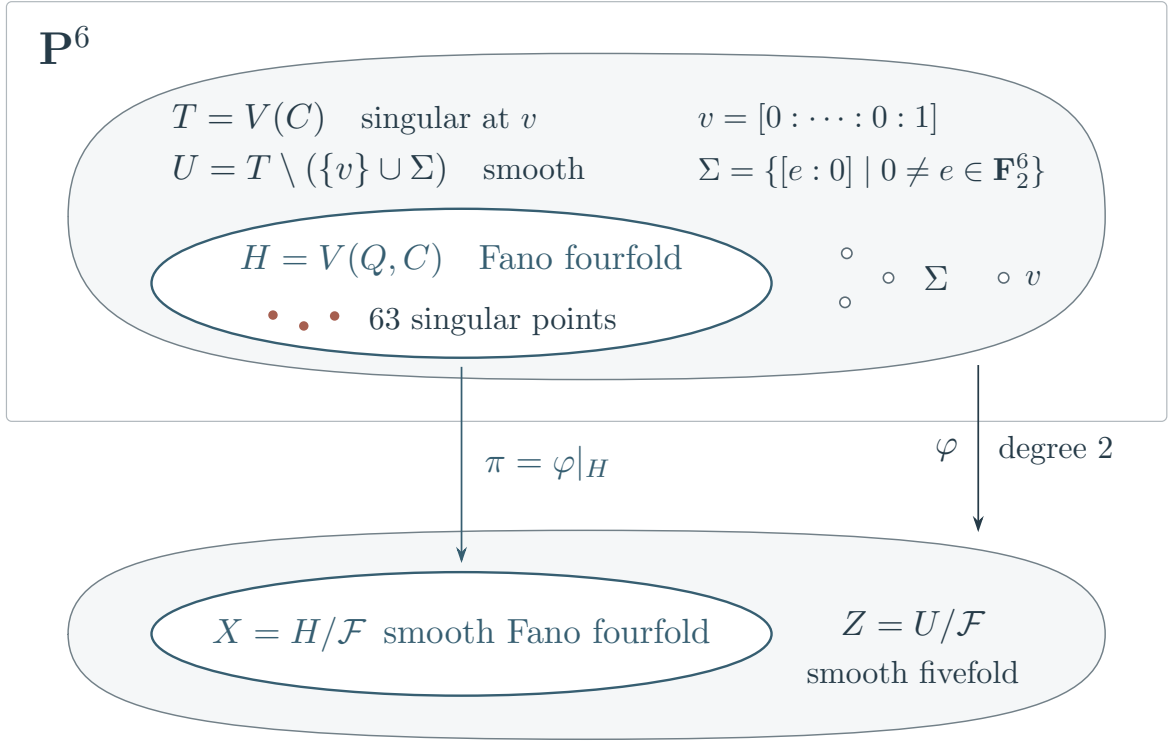
\[
\begin{tikzpicture}[
  scale=1.2,
  transform shape,
  x=1cm,y=1cm,
  every node/.style={font=\small,text=ink,inner sep=2pt},
  ambient/.style={draw=ink!36,line width=.45pt,rounded corners=2pt},
  region/.style={draw=ink!65,fill=paperwash,line width=.55pt},
  divisor/.style={draw=accent,fill=white,line width=.9pt},
  quotient/.style={-{Stealth[length=2.15mm,width=1.45mm]},draw=ink,line width=.65pt},
  missing/.style={draw=ink!75,fill=white,line width=.6pt}
]
    \draw[ambient] (0,2.50) rectangle (12.8,7.12);
  \node[anchor=north west,font=\large] at (.28,6.92) {$\mathbf P^6$};

    \path[region]
    (.68,4.75)
    .. controls (.68,6.15) and (2.05,6.55) .. (6.40,6.55)
    .. controls (11.15,6.55) and (12.12,6.35) .. (12.12,4.75)
    .. controls (12.12,3.30) and (10.54,2.96) .. (6.40,2.96)
    .. controls (2.25,2.96) and (.68,3.30) .. cycle;
  \node[anchor=west] at (1.75,5.83)
    {$T=V(C)$\quad {\footnotesize singular at $v$}};
  \node[anchor=west] at (1.75,5.26)
    {$U=T\setminus(\{v\}\cup\Sigma)$\quad {\footnotesize smooth}};
  \node[anchor=west,font=\footnotesize] at (7.55,5.83)
    {$v=[0:\cdots:0:1]$};
  \node[anchor=west,font=\footnotesize] at (7.55,5.26)
    {$\Sigma=\{[e:0]\mid 0\ne e\in\mathbf F_2^6\}$};

    \draw[divisor] (5.02,4.02) ellipse [x radius=3.42,y radius=.82];
  \node[text=accent] at (5.02,4.27)
    {$H=V(Q,C)$\quad Fano fourfold};
  \foreach \x/\y in {2.94/3.67,3.28/3.55,3.62/3.66}
    \fill[singularpoint] (\x,\y) circle[radius=.052];
  \node[anchor=west,font=\footnotesize] at (3.92,3.61)
    {63 singular points};

    \draw[missing] (11.00,4.08) circle[radius=.057];
  \node[anchor=west,font=\footnotesize] at (11.16,4.08) {$v$};
  \foreach \x/\y in {9.27/4.35,9.73/4.08,9.25/3.81}
    \draw[missing] (\x,\y) circle[radius=.057];
  \node[anchor=west] at (10.04,4.08) {$\Sigma$};

    \path[region]
    (.68,.16)
    .. controls (.68,1.05) and (2.15,1.30) .. (6.40,1.30)
    .. controls (10.64,1.30) and (12.12,1.05) .. (12.12,.16)
    .. controls (12.12,-.73) and (10.63,-1.00) .. (6.40,-1.00)
    .. controls (2.16,-1.00) and (.68,-.73) .. cycle;
  \draw[divisor] (5.02,.16) ellipse [x radius=3.42,y radius=.68];
  \node[text=accent] at (5.02,.16)
    {$X=H/\mathcal F$\enspace smooth Fano fourfold};
  \node[align=center] at (10.00,.03)
    {$Z=U/\mathcal F$\\[2pt] {\footnotesize smooth fivefold}};

    \draw[quotient,draw=accent] (5.02,3.10) -- (5.02,.94)
    node[pos=.49,anchor=west,xshift=5pt,text=accent] {$\pi=\varphi|_H$};
  \draw[quotient] (10.72,3.12) -- (10.72,1.28)
    node[pos=.51,anchor=east,xshift=-4pt] {$\varphi$}
    node[pos=.51,anchor=west,xshift=4pt,font=\footnotesize] {degree $2$};
\end{tikzpicture}
\]

\subsection*{AI statement}
The example in this article was discovered with assistance from
ChatGPT~6 Astra. ChatGPT was subsequently also used to edit
the LaTeX manuscript. The authors take full responsibility for the contents.

\subsection*{Acknowledgements}
Joe Waldron is supported by NSF CAREER grant 2440240 and NSF grant 2401279. Jakub Witaszek is supported by NSF CAREER grant 2540921 and NSF grant 2401360.

\DIFaddend \section{Regular foliations in characteristic two}\label{sec:foliations}

Throughout this section, we fix $k :=\overline{\bF}_2$. For a smooth $k$-variety $U$, we write
$T_U:=\mathcal{D}er_k(\cO_U,\cO_U)$ and
$\omega_U:=\det\Omega^1_{U/k}$.
\begin{definition}\label{def:foliation}
A \emph{regular rank-one foliation} on $U$ is a line subbundle
$\cF\subset T_U$, with locally free quotient, closed under Lie
brackets and squaring of derivations.
\end{definition}

Here the Lie bracket is defined by
$[\delta,\delta'](f) :=\delta(\delta'(f))-\delta'(\delta(f))$, and
the square of a derivation $\delta$ means $\delta^2(f):=\delta(\delta(f))$. Both operations produce
derivations in characteristic two. For a local generator $\delta$ of
$\cF$ and local functions $a,b$, one has
\[
 [a\delta,b\delta]=(a\delta(b)-b\delta(a))\delta,
 \qquad (a\delta)^2=a\delta(a)\delta+a^2\delta^2.
\]
Thus bracket closure is automatic in rank one, and to verify the square closure it is enough to check it  on local generators $\delta$. Here,  the regularity means that a local
generator is nonzero in every fibre of the tangent sheaf.

\begin{theorem}[{\cite{Ekedahl}*{I, Proposition~1.1}}]
\label{thm:ekedahl}
Let $\cF\subset T_U$ be a regular rank-one foliation on a smooth
$k$-variety $U$. The ringed space with underlying space $U$ and
structure sheaf $\cO_Z :=\cO_U^{\cF}$, consisting of functions
annihilated by every local section of $\cF$, is a smooth $k$-variety,
denoted $Z=U/\cF$. The natural morphism
$\varphi \colon U\to Z$ is a finite flat  universal homeomorphism of degree two.
The absolute Frobenius factors as
\[F \colon U\xrightarrow{\varphi}Z\xrightarrow{\psi}U,
\]
where $\psi$ is also finite flat. There is an exact sequence
\[
 0\to\cF\to T_U\to \varphi^*T_Z\to\cF^{\otimes2}\to0,
\]
and a canonical isomorphism 
\begin{equation}\label{eq:canonical-pullback}
 \varphi^*\omega_Z\simeq\omega_U\otimes\cF^{-1}.
\end{equation}
\end{theorem}

\begin{remark}\label{rem:local-quotient}
We will use the following local description from
\cite{Ekedahl}*{I, Proposition~1.1(ii)}. Locally on $Z$, write
 $\varphi^{-1}(\Spec B)=\Spec A$. There is an element $v\in A$ with
$v^2\in B$ such that
\begin{equation}\label{eq:local-cover}
 A=B\oplus Bv\qquad\text{as }B\text{-modules}.
\end{equation}
On $\Spec A$, the foliation is generated by the $B$-linear
derivation $\partial:A\to A$ defined by
$\partial(a+cv)=c$ for $a,c\in B$.
Writing $d=\dim U$, in suitable \'etale coordinates the quotient
map is given by
\[
 (v,x_2,\ldots,x_d)\longmapsto(v^2,x_2,\ldots,x_d).
\]
\end{remark}

We will restrict $\varphi$ to a closed subscheme of $U$ preserved by the foliation. The
following lemma explains why that restriction remains a finite flat
quotient.

\begin{lemma}\label{lem:invariant-subschemes}
In the setting of Theorem~\ref{thm:ekedahl}, let
$\cI\subset\cO_U$ be a coherent ideal sheaf preserved by every local
section of $\cF$. There is a unique ideal sheaf $\cJ\subset\cO_Z$
such that $\cI=\cJ\cO_U$. If $V:=V(\cI)$ and $W:=V(\cJ)$, then
\[
 V=U\times_Z W.
\]
The induced morphism $V\to W$ is a finite flat universal homeomorphism of degree
two, and $\cO_W$ is the sheaf of invariant functions on $V$ under
the restricted derivations.
\end{lemma}

\begin{proof}
Use the local description in Remark~\ref{rem:local-quotient}, and let
$I\subset A$ be the ideal defining $V$. Set $J :=I\cap B$. We can write an arbitrary element of $I$ as  $a+cv$ with $a,c\in B$. Then invariance under $\partial$ gives
$c=\partial(a+cv)\in J$. Subtracting $cv$ gives $a\in J$.
Consequently,
\[
 I=JA=J\oplus Jv \qquad
 A/I=(B/J)\oplus(B/J)v.
\]
The invariants of the last ring are exactly $B/J$. The ideals
$J=I\cap B$ glue on $Z$ and are uniquely determined by $I$.
The remaining assertions follow by base change from  $\varphi$.
\end{proof}

\section{The construction}\label{sec:construction}

Below we explain the key construction of the paper.
\begin{construction}\label{construction:example}
Let $\theta,Q_6,Q,C$ be as in Section~\ref{sec:introduction}, and set
\[
 S :=k[x_0,\ldots,x_5,t] \qquad R :=S/(Q,C).
\]
Write $S_n$ and $R_n$ for their degree-$n$ parts, and put
$T :=V(C)\subset\bP^6$ and $H :=V(Q,C)\subset T$.
Consider the degree-one derivation
\begin{equation}\label{eq:derivation}
 D:=
  \sum_{i=0}^5x_i^2\frac{\partial}{\partial x_i}.
\end{equation}
Thus $D(x_i)=x_i^2$, $D(t)=0$, and $D(S_n)\subset S_{n+1}$.
\end{construction}

The coefficients of $Q_6$ ensure that, for every
$0\ne e=(e_0,\ldots,e_5)\in\bF_2^6$,
\begin{equation}\label{eq:binary}
 Q_6(e)=\sum_{j=0}^2e_{2j}e_{2j+1}
          +\sum_{i=0}^5e_i\theta^{i+1}\ne0.
\end{equation}
Indeed, $1,\theta,\ldots,\theta^6$ are linearly independent over
$\bF_2$, and some $e_i$ is nonzero.

\begin{proposition}\label{prop:complete-intersection}
The hypersurface $T$ is smooth away from its vertex
$[0:\cdots:0:1]$. The scheme $H$ is a normal integral complete
intersection of dimension four. Its singular points are exactly
\begin{equation}\label{eq:singular-points}
 P_e=[e:s_e],\qquad
 0\ne e\in\bF_2^6,\qquad s_e^2=Q_6(e).
\end{equation}
In particular, there are $63$ singular points, and each $s_e$ is
nonzero.
\end{proposition}

\begin{proof}
The polynomial $C$ is nonzero, and $Q=Q_6+t^2$ is a
nonzerodivisor modulo $C$ because it is monic in $t$. Thus
$(C,Q)$ is a regular sequence in $S$, and $H$ is a complete
intersection of dimension four.

The gradients are
\begin{align*}
 \nabla Q&=(x_1,x_0,x_3,x_2,x_5,x_4,0),\\
 \nabla C&=(x_1^2,x_0^2,x_3^2,x_2^2,x_5^2,x_4^2,0).
\end{align*}
Thus $T$ is smooth outside its vertex. Since $Q=1$ at the vertex,
both gradients are nonzero on $H$.

A singular point of $H$ is characterised uniquely by the fact that the two nonzero gradients are dependent.
Thus $x_i^2=\lambda x_i$ for every $i$, with $\lambda\ne0$.
Rescaling the coordinates by $\lambda^{-1}$ makes
$(x_0,\ldots,x_5)$ a nonzero binary vector $e \in \bF^6_2$. Conversely, $C(e)=0$
for every such vector, and the gradients are dependent at
$[e:s_e]$. The equation $Q=0$ determines $s_e$ uniquely, and
\eqref{eq:binary} gives $s_e\ne0$. This proves the assertion about
the singular points.

As a positive-dimensional complete intersection in projective space, $H$ is connected and Cohen--Macaulay.
 It is regular
in codimension one, since its singular locus is finite; hence it is
normal by Serre's criterion. The irreducible components of a normal
scheme are disjoint, so connectedness implies that $H$ is integral.
\end{proof}

\begin{proposition}\label{prop:foliation}
Let
\[
 U:=T\setminus\left(\bigl\{[0:\cdots:0:1]\}
               \cup\{[e:0]\mid 0\ne e\in\bF_2^6\bigr\}\right).
\]
Then $U$ is smooth and contains $H$. Moreover, it carries a regular rank-one foliation $\cF \cong \cO_U(-1)$ induced by $D$, which restricts to a foliation on $H$.
\end{proposition}

\begin{proof}
The preceding proposition shows that $U$ is smooth. The vertex does
not lie on $V(Q)$, and \eqref{eq:binary} excludes each $[e:0]$.
Thus $H\subset U$.

\vspace{0.7em}
\noindent\textbf{Descent to $T$ and $H$.}
Direct differentiation gives
\begin{equation}\label{eq:derivation-identities}
 D(Q)=C,\qquad D(C)=0,\qquad D^2=0.
\end{equation}
Thus $D$ induces a degree-one
derivation on  $S/(C)$ and $S/(Q,C)$.
\DIFaddbegin

\vspace{0.7em}
\noindent\textbf{Identification with $\cO_U(-1)$.}
 On the chart $T\,\cap\,\{x_i\ne0\}$, $D$ extends
to homogeneous localizations by the quotient rule, and
\[
 \delta_i:=x_i^{-1}D\colon\,
 (S/(C))[x_i^{-1}]_0\longrightarrow (S/(C))[x_i^{-1}]_0
\]
is a derivation of its coordinate ring. Restrict it to
$U_i :=U\cap\{x_i\ne0\}$. These opens cover $U$.

On $U_i\cap U_j$,
\[
 \delta_j=(x_i/x_j)\delta_i.
\]
These are the transition functions of $\cO_U(-1)$, so the vector
fields give an isomorphism $\cF\cong\cO_U(-1)$.

\vspace{0.7em}
\noindent\textbf{Nonvanishing and regularity.}
For the affine coordinate functions
$r_j :=x_j/x_i$ with $j\ne i$, and $u=t/x_i$, one has
\begin{equation}\label{eq:local-action}
 \delta_i(r_j)=r_j^2+r_j\qquad \delta_i(u)=u.
\end{equation}
The coordinate functions $r_j,u$ generate the coordinate ring of
the affine chart of $T$. Thus $\delta_i$ vanishes at a point
precisely when all the values in \eqref{eq:local-action} vanish.
This forces $r_j\in\bF_2$ and $u=0$, thus giving one of the removed
points $[e:0]$.
The local generators of $\cF$ are therefore nowhere
vanishing, so $\cF\subset T_U$ is  a line subbundle with locally free quotient.

\vspace{0.7em}
\noindent\textbf{Closure under squaring.}
Equation~\eqref{eq:local-action} gives $\delta_i^2=\delta_i$ on
the coordinate functions, hence on their ring and its localizations.
The criterion following Definition~\ref{def:foliation} completes
the proof.
\end{proof}

We now take the quotient by the foliation studied in Proposition \ref{prop:foliation}.

\DIFaddend \begin{proposition}\label{prop:quotient}
Let $\varphi \colon U\to Z :=U/\cF$  be the quotient  by the foliation of
Proposition~\ref{prop:foliation}. There is a Cartier divisor
$X\subset Z$ with $\varphi^{-1}(X)=H$. The fourfold $X$ is projective  and
integral, $Z$ is smooth, and the restriction $\pi\colon H\to X$ is a
finite flat universal homeomorphism of degree two.
Moreover, $\cO_X$ is the sheaf of invariant functions on $H$.
\end{proposition}

\begin{proof}
Theorem~\ref{thm:ekedahl} provides us with the smooth quotient $Z$. By the invariance established in Proposition~\ref{prop:foliation}, Lemma~\ref{lem:invariant-subschemes}
gives $X$ and the Cartesian square
\[
\begin{tikzcd}[column sep=large,row sep=large]
 H \ar[r,hook] \ar[d,"\pi"'] & U \ar[d,"{\varphi}"] \\
 X \ar[r,hook] & Z,
\end{tikzcd}
\]
together with the assertions about $\pi$ and invariant functions. Since $\varphi$ is faithfully flat and its inverse image
$H$ is an effective Cartier divisor in $U$, the closed subscheme $X$ is an effective Cartier divisor in $Z$; see 
\stacksproj{05B2}. The remaining assertions are clear by construction. \qedhere

\end{proof}

By Ekedahl's theorem~\ref{thm:ekedahl} and Proposition \ref{prop:foliation}, $X$ is smooth away from the images of the
$63$ singular points of $H$. To prove smoothness at the image of each singular point
$P_e=[e:s_e]$, we consider
a line $\Gamma_e\subseteq T=V(C)$ through $P_e$ parametrised  by $t \mapsto [e:t]$. Working on an affine chart,
we show that the local equation defined by $Q$ descends to a function
whose restriction to $\varphi(\Gamma_e)$ is unramified at $\varphi(P_e)$. By the Jacobian criterion, this shows that $X$ is smooth at $\phi(P_e)$.
\begin{proposition}\label{prop:smooth}
The fourfold $X$ is smooth.
\end{proposition}

\begin{proof}
As remarked above, it suffices to
prove smoothness at $y:=\pi(P)$, where $P=P_e=[e:s_e]$ are the singular points of $H$ determined in Proposition \ref{prop:complete-intersection}.
Fix such a point $P$, and choose $i$ with $e_i=1$.

Set $V_i:=\varphi(U_i)$. We shall use the
coordinates $r_j,u$ of \eqref{eq:local-action}. In particular, 
$r_j(P)=e_j$ and $u(P)=s_e\ne0$. Moreover, 
let $f_i:=Q/x_i^2$ and $g_i$ be local equations of $H$ on $U_i$
and $X$ on $V_i$, respectively, chosen so that
\begin{equation}\label{eq:local-equation}
 \varphi^*g_i=f_i.
\end{equation}
Such equations exist by Proposition \ref{prop:quotient}.
Since $Z$ is smooth and $X$ has equation $g_i$ on $V_i$, it  suffices to show that $dg_i(y)\ne0$.

Since  the equation $C$ does not contain the variable $t$ and $C(e)=0$, the affine line
parametrised by $t\mapsto[e:t]$ lies in $T$. Its intersection
with $U_i$ is the closed curve through $P$:
\[
 \Gamma_e:=\{r_j=e_j\text{ for }j\ne i\}
          \simeq\Spec k[u,u^{-1}]\subset U_i.
\]
The point $u=0$ was removed in defining $U$.
The ideal of $\Gamma_e$ is invariant because
\[
 \delta_i(r_j-e_j)=(r_j-e_j)^2+(r_j-e_j).
\]
On $\Gamma_e$, the derivation is $u\partial_u$, whose invariant
ring is $k[u^2,u^{-2}]$. Lemma~\ref{lem:invariant-subschemes}
therefore gives a closed curve through $y$,
\[
 \overline\Gamma_e=\Spec k[b,b^{-1}]\subset V_i,
\]
with pullback $b\mapsto u^2$. Since $\varphi^*g_i=f_i$ and
$Q_6(e)=s_e^2$, we obtain
\[
 f_i|_{\Gamma_e}=(u-s_e)^2,
 \qquad g_i|_{\overline\Gamma_e}=b-s_e^2.
\]
The second equation has a simple zero at $y$, so $dg_i(y)\ne0$.
\end{proof}

\begin{proposition}\label{prop:fano}
The fourfold $X$ is Fano, and
\begin{equation}\label{eq:canonical-example}
 \pi^*\omega_X\simeq\cO_H(-1).
\end{equation}
\end{proposition}
\begin{proof}
Adjunction for the cubic hypersurface gives
$\omega_U\simeq\cO_U(-4)$. Since $\cF\simeq\cO_U(-1)$,
the canonical-bundle formula \eqref{eq:canonical-pullback} gives
\[
\varphi^*\omega_Z\simeq\omega_U\otimes\cF^{-1}
                  \simeq\cO_U(-3).
\]
By Proposition~\ref{prop:quotient},  $\varphi^*X=H$ as Cartier divisors,
and $\cO_U(H)\simeq\cO_U(2)$ because $H$ is cut out by $Q$.
Adjunction for the smooth divisor $X\subset Z$ now yields
\[
\pi^*\omega_X
 \simeq\bigl(\varphi^*\omega_Z\otimes\cO_U(H)\bigr)|_H
 \simeq\cO_H(-1).
\]
Thus $\pi^*\omega_X^{-1}\simeq\cO_H(1)$ is ample. Since $\pi$
is finite and surjective and both varieties are proper,
$\omega_X^{-1}$ is ample by \stacksproj{0B5V}. Thus $X$ is Fano.
\end{proof}

\section{The cohomology calculation}\label{sec:cohomology}

We retain the construction and notation $\pi \colon H\to X$ from Proposition~\ref{prop:quotient}.
Recall from Construction~\ref{construction:example} that
$S=k[x_0,\ldots,x_5,t]$, $R=S/(Q,C)$, and $S_n,R_n$ denote the degree-$n$ parts of $S$ and $R$, respectively. In particular, $S_1$ is the space of linear
forms, and $R_2=S_2/kQ$, where $kQ=\{aQ\mid a\in k\}$.

\begin{lemma}\label{lem:koszul}
Restriction induces
\[
 H^0(H,\cO_H)=k\qquad
 H^0(H,\cO_H(1))=R_1=S_1\qquad
 H^0(H,\cO_H(2))=R_2.
\]
Moreover, $H^j(H,\cO_H)=0$ for every $j>0$.
\end{lemma}

\begin{proof}
Consider the twisted Koszul resolution 
\begin{equation}\label{eq:koszul}
\begin{split}
 0\to\cO_{\bP^6}(n-5)
 &\to\cO_{\bP^6}(n-2)\oplus\cO_{\bP^6}(n-3)\\
 &\to\cO_{\bP^6}(n)\to\cO_H(n)\to0.
\end{split}
\end{equation}
For $n=0,1,2$, all projective-space terms in the Koszul resolution
\eqref{eq:koszul} have vanishing higher cohomology. At $n=0$,
the negative twists have no global sections either, giving
$H^0(H,\cO_H)=k$ and the asserted vanishing. At $n=1,2$, taking
global sections gives $S_1$ and $S_2/kQ$, respectively.
\end{proof}

The next lemma supplies the short exact sequence used to compute
$H^1(X,\cO_X)$. Its maps come directly from the polynomial derivation.

\begin{lemma} \label{lem:derivation-sequence}
The derivation $D$ induces an  exact complex of $\cO_X$-modules
\begin{equation} \label{eq:DComplex}
 0\to\cO_X\to\pi_*\cO_H\xrightarrow{D}\pi_*\cO_H(1)
                   \xrightarrow{D}\pi_*\cO_H(2).
\end{equation}
Equivalently, for
\[
K :=\ker\bigl(D\colon \pi_*\cO_H(1)\to\pi_*\cO_H(2)\bigr),
\]
there is a short exact sequence
\begin{equation}\label{eq:derivation-sequence}
 0\to\cO_X\to\pi_*\cO_H\xrightarrow{D}K\to0.
\end{equation}
\end{lemma}

\begin{proof}We split the proof into two parts.\\[-0.5em]

\noindent\textbf{Construction of Complex (\ref{eq:DComplex})}\ 
By \eqref{eq:derivation-identities}, $D$ induces a degree-one
derivation on $R=S/(Q,C)$ with $D^2=0$. Extending it to homogeneous
localizations by the quotient rule and taking degree-$n$ parts gives
compatible maps $\cO_H(n)\to\cO_H(n+1)$. These maps satisfy the
Leibniz rule for products of sections.

On $H\cap U_i$, differentiation of a regular function $f$ is
$D(f)=x_i\delta_i(f)$, a section of $\cO_H(1)$. Hence $D$ kills
functions pulled back from $X$, so the displayed maps
are $\cO_X$-linear. Since $D^2=0$, they form  a complex. By Proposition~\ref{prop:quotient}, this complex is exact at $\cO_X$ and $\pi_*\cO_H$.

\vspace{0.7em}
\noindent\textbf{Surjectivity onto $K$ in (\ref{eq:derivation-sequence})}.
The opens $H\cap U_i$ cover $H$ and correspond to an open cover
of $X$, since $\pi$ is a homeomorphism. On one of these opens,
let $s$ be a local section of $\cO_H(1)$ satisfying $D(s)=0$.
Since $x_i$ is a nowhere-vanishing section of $\cO_H(1)$ there,
$s/x_i$ is a regular function. The quotient rule gives
$D(x_i^{-1})=x_i^{-2}D(x_i)=1$ in characteristic two. Consequently,
\[
 D(s/x_i)=D(s)/x_i+s=s.
\]
Every local section of $K$ therefore has a local preimage. This
proves the required surjectivity of sheaves and concludes the proof that  (\ref{eq:derivation-sequence}), equivalently (\ref{eq:DComplex}), are exact.
\end{proof}

\begin{proposition} \label{prop:cohomology}
One has $H^1(X,\cO_X)\simeq k$.
\end{proposition}

\begin{proof}
Since $\pi$ is finite, Lemma~\ref{lem:koszul} gives
\[
 H^0(X,\pi_*\cO_H)=k,\qquad H^1(X,\pi_*\cO_H)=0.
\]
The inclusion $\cO_X\to\pi_*\cO_H$ induces an isomorphism on
global sections.
Thus the long exact sequence associated to
\eqref{eq:derivation-sequence}, followed by the definition of $K$,
gives
\[
 H^1(X,\cO_X)\simeq H^0(X,K)
                =\ker\bigl(D:R_1\to R_2\bigr).
\]

We claim that the kernel is $k \cdot t$. For a linear form
$\ell=\sum_{i=0}^5a_ix_i+ct$, we have
$D(\ell)=\sum_{i=0}^5a_ix_i^2$. If its image in $R_2=S_2/kQ$
is zero, then
\[
 \sum_{i=0}^5a_ix_i^2=\lambda Q\qquad\text{in }S_2
\]
for some $\lambda\in k$. The coefficient of $t^2$ gives
$\lambda=0$, and the coefficients of $x_i^2$ then give $a_i=0$
for every $i$. Conversely, $D(t)=0$. Hence the kernel is indeed $k\cdot t$,
which is one-dimensional.
\end{proof}

\begin{remark} \label{rem:FrobeniusZero}
We note that the absolute Frobenius acts as zero on $H^1(X,\cO_X)$. To see this, we first note that the square of a local function on $H$ is invariant under the foliation, so we get a factorisation:  \[F\colon X\xrightarrow{}H\xrightarrow{\pi}X.\] Since $\pi$ is finite, Lemma~\ref{lem:koszul} gives
$H^1(X,\pi_*\cO_H)=0$.
Hence the composite map
\[
 F^* \colon H^1(X,\cO_X)\to H^1(X,\pi_*\cO_H) \to H^1(X,F_*\cO_X)
\]
is zero as required.
\end{remark}

 \begin{proof}[Proof of Theorem~\ref{thm:main}]
 Combine Propositions~\ref{prop:quotient}, \ref{prop:smooth},
\ref{prop:fano}, and~\ref{prop:cohomology}.
\end{proof}

As it will be needed in the next section, we provide more information on the sheaf $K$ from Lemma \ref{lem:derivation-sequence}.
\begin{proposition} \label{prop:K}
With notation as in Lemma \ref{lem:derivation-sequence}, we have that $H^0(X,K) \neq 0$, that $K$ is an ample line bundle, and that the inclusion
$K\subset\pi_*\cO_H(1)$ induces an isomorphism
\begin{equation}\label{eq:kernel-pullback}
 \pi^*K\simeq\cO_H(1).
 \end{equation}
\end{proposition}
\begin{proof}
The fact that $H^0(X,K) \neq 0$ is immediate from (\ref{eq:derivation-sequence}) given that $H^1(X,\pi_*\cO_H)=0$ and $H^1(X,\sO_X)\neq 0$. The remaining statements are verified using a local computation.

By Remark~\ref{rem:local-quotient} and
Lemma~\ref{lem:invariant-subschemes}, locally on $X$ we may write
$\pi^{-1}(\Spec B)=\Spec A$ with $A=B\oplus Bv$ and $v^2\in B$.
The exact sequence (\ref{eq:derivation-sequence}) identifies $\pi_*\cO_H/\cO_X$ with $K$ via
$[a]\mapsto D(a) \in K$. Thus $K$ is a line bundle, locally generated
by the class $[v]$.

Shrinking $X$ so that $\Spec A\subset H\cap U_i$, write
$\delta_i=c\partial$ with $c\in A^\times$ and $\partial(v)=1$.
Here $c$ is a unit because $\delta_i$ and $\partial$ come from
local generators of the foliation on $U$. By taking the adjoint of the natural inclusion $K \hookrightarrow \pi_*\cO_H(1)$, we get the homomorphism
$\pi^*K\to\cO_H(1)$ sending $[v]$ to $D(v)=x_ic$, which is
nowhere vanishing. This proves \eqref{eq:kernel-pullback}.
Since $\pi$ is finite and surjective, ampleness of $\cO_H(1)$
implies ampleness of $K$ by 
\stacksproj{0B5V}.   
\end{proof}

\section{Differential forms and rational curves}\label{sec:rational-curves}

In this section we construct an embedding of the ample line bundle $K^{\otimes 2}$ inside $\Omega^1_{X/k}$. We then use it to construct a global one-form on $X$ and show that $X$ is not separably uniruled.

Let $F\colon X\to X$ denote absolute Frobenius. As it will be used in the proof, we define the
sheaf of locally exact one-forms by
\[
 B\Omega^1_X :=\operatorname{im}\bigl(
 d\colon F_*\cO_X\to F_*\Omega^1_{X/k}\bigr).
\]
Here $d$ is $\cO_X$-linear, so $B\Omega^1_X$ is an $\cO_X$-module. There is a natural short exact sequence
\begin{equation} \label{eq:Bseq}
0 \to \cO_X \to F_*\cO_X \to B\Omega_X^1\to 0.
\end{equation}

\begin{proposition}\label{prop:frobenius-one-forms}
 There exists an injective map
\begin{equation}\label{eq:positive-differential}
 \beta:K^{\otimes2}\longrightarrow\Omega^1_{X/k}
\end{equation}
that is nonzero on every fibre outside $\pi(\operatorname{Sing}H)$.  In particular, $H^0(X,\Omega^1_{X/k}) \neq 0$.

\end{proposition}
\begin{proof} The final assertion is immediate from the fact that $K^{\otimes 2}$ embeds into $\Omega^1_{X/k}$ and that $H^0(X, K) \neq 0$ (see Proposition \ref{prop:K}).\\[-0.5em]

\noindent\textbf{The existence of $\beta$.}
Squares of local functions on $H$ are invariant under the foliation, so they descend to functions on $X$. Using the identification  of $K$ with $\pi_*\cO_H/\cO_X$, we define the map
\[
 \eta \colon K\longrightarrow B\Omega_X^1,\qquad [a]\longmapsto d(a^2),
\]
where $a$ is a local function on $H$ and $a^2$ is regarded as a
function on $X$. Replacing $a$ by $a+c$ for $c\in\cO_X$ changes
$a^2$ by $c^2$, whose differential is zero. Thus $\eta$ is
well-defined. The identity $d((ca)^2)=c^2d(a^2)$ shows that
$\eta$ is $\cO_X$-linear.
Composing with $B\Omega_X^1\subset F_*\Omega^1_{X/k}$ and applying adjunction gives \eqref{eq:positive-differential},  since
$F^*K\simeq K^{\otimes2}$.

\vspace{0.7em}
\noindent\textbf{Fibrewise nonvanishing.}
Use the local description $A=B\oplus Bv$ from
Proposition~\ref{prop:K}, and put $b:=v^2\in B$.
Then
\[
 A\simeq B[z]/(z^2-b),\qquad
 \beta([v]^{\otimes2})=db.
\]
Here $[v]$ is a local generator of $K$. Since $X$ is smooth and
$d(z^2-b)=-db$ in characteristic two, the Jacobian
criterion shows that $db$ is nonzero at a point of $X$ exactly
when $H$ is smooth at the unique point above it. This proves the
assertion about  $\beta$ being non-zero on fibres.
\end{proof}

The existence of a non-zero global one-form already implies that $X$ is not separably
rationally connected (see \cite{Gounelas}).  Below we show that $X$ is not even separably uniruled.

A nonconstant morphism $f \colon \bP^1\to X$ is called \emph{free} if
$f^*T_X$ is globally generated. A smooth projective variety over an
algebraically closed field is separably uniruled if and only if it
admits a free rational curve, see \cite{Sato}*{Proposition~1.1}.

\begin{proposition}\label{prop:not-separably-uniruled}
The fourfold $X$ admits no free rational curve. In particular,
it is not separably uniruled.
\end{proposition}

\begin{proof}
Let $\beta$ be the homomorphism of
Proposition~\ref{prop:frobenius-one-forms}, and suppose that
$f:\bP^1\to X$ is free. By
Proposition~\ref{prop:complete-intersection}, the set
$\pi(\operatorname{Sing}H)$ is finite. Since $f$ is nonconstant,
its image meets the complement of this set, so $f^*\beta\ne0$.
Freeness gives a surjection
$\cO_{\bP^1}^{\oplus N}\to f^*T_X$ for some $N>0$, and
dualizing yields
\[
 f^*K^{\otimes2}\xrightarrow{f^*\beta}f^*\Omega^1_{X/k}
                 \hookrightarrow\cO_{\bP^1}^{\oplus N}.
\]
Since the composition of these maps is nonzero, we get a contradiction with the fact that $K$ is ample (Proposition~\ref{prop:K}).
\end{proof}

 \DIFaddend

\end{document}